\documentclass{amsart}
\numberwithin{equation}{section}
\newtheorem{theorem}{Theorem}[section]

\newtheorem{proposition}[theorem]{Proposition}
\newtheorem{conjecture}[theorem]{Conjecture}
\newtheorem{lemma}[theorem]{Lemma}
\theoremstyle{definition}

\newtheorem{example}[theorem]{Example}

\theoremstyle{remark}
\newtheorem{remark}[theorem]{\bf\em Remark}

\usepackage{graphicx}
\usepackage{amsfonts}
\usepackage{amssymb}
\usepackage{amsmath}
\usepackage{bbm}
\usepackage[inner=1in,outer=1in,bottom=1in,top=1in]{geometry}

\newcommand{\G}{\mathrm{G}}
\newcommand{\s}{\mathrm{S}}

\newcommand{\leg}[2]{\left(\frac{#1}{#2}\right)}
\newcommand{\Tr}{\mathrm{Tr}}
\newcommand{\Sym}{\mathrm{Sym}}

\def\+{{\oplus}}

\makeatletter
\def\Ddots{\mathinner{\mkern1mu\raise\p@
\vbox{\kern7\p@\hbox{.}}\mkern2mu
\raise4\p@\hbox{.}\mkern2mu\raise7\p@\hbox{.}\mkern1mu}}
\makeatother

\begin{document}
\title[Value distribution of elementary of symmetric polynomials]{Value distribution of elementary symmetric polynomials and its perturbations over finite fields}

\author{Luis A. Medina}
\address{Department of Mathematics, University of Puerto Rico, 17 Ave. Universidad STE 1701, San Juan, PR 00925}
\email{luis.medina17@upr.edu}

\author{L. Brehsner Sep\'ulveda}
\address{Department of Mathematics, University of Puerto Rico, 17 Ave. Universidad STE 1701, San Juan, PR 00925}
\email{leonid.sepulveda1@upr.edu}

\author{C\'esar A. Serna-Rapello}
\address{Department of Mathematics, University of Puerto Rico, 17 Ave. Universidad STE 1701, San Juan, PR 00925}
\email{cesar.serna@upr.edu}

\begin{abstract}
In this article we establish the asymptotic behavior of generating functions related to the exponential sum over finite fields of elementary symmetric functions and their perturbations.  This asymptotic behavior allows us to calculate the probability generating function of the probability that the the elementary symmetric polynomial of degree $k$ and its perturbations returns $\beta \in \mathbb{F}_q$ where $\mathbb{F}_q$ represents the field of $q$ elements.  Our study extends many of the results known for perturbations over the binary field to any finite field.  In particular, we establish when a particular perturbation is asymptotically balanced over a prime field and provide a construction to find such perturbations over any finite field.

\end{abstract}

\date{\today}
\maketitle

{\centering\footnotesize In memory of Francis N. Castro.\par}
\section{Introduction}

Many problems in number theory and combinatorics, as well as in their applications, can be formulated in terms of exponential sums.  In cryptography, for example, exponential sums can be used to detect when 
a particular function is balanced (a property very useful in cryptographic applications) \cite{cgm2,cm1,cm2,cm3,cms,cusick4,cusick2}.  Some classical examples of exponential sums include the number-theoretical
Gauss sums, Kloosterman sums, and Weyl sums. 

This work is based on the study of exponential sums of the following form.  Let $q=p^r$ where $p$ is prime and $r\geq 1$.  Let $F:\mathbb{F}_q^n \to \mathbb{F}_q$ be a function.  The {\em exponential sum over }$\mathbb{F}_q$ of $F$ is defined as
\begin{equation}
\label{expsumFq}
S_{\mathbb{F}_q}(F)=\sum_{{\bf x}\in\mathbb{F}_q^n}\xi_p^{\Tr(F({\bf x}))},
\end{equation}
where $\xi_p=\exp(2\pi i/p)$ and $\Tr = \Tr_{\mathbb{F}_q/\mathbb{F}_p}$ is the field trace function.  These exponential sums have been extensively studied when the characteristic of the field is 2 because 
of their cryptographic applications, see  \cite{cai,canteaut, cgm2,cm3,cms,cusick4,cusick2,mitchell}.  
Recently, some cryptographic applications when the characteristic of the field is different than 2 has been found.  This has prompted new research in exponential sums of the type (\ref{expsumFq}) and many of the results available for the binary field have been extended to other finite fields \cite{ccms,cmsep,licusick1,licusick2,llm}.  

Let $L:\mathbb{F}_q\to \mathbb{F}_q$ be a linear function and $X$ an indeterminate.
Consider the generating function given by
\begin{equation}
\label{genfundef}
 \s_{\mathbb{F}_q,L}(F;X)=\sum_{{\bf y} \in \mathbb{F}_q^n}X^{L(F({\bf y}))}.
\end{equation}
Observe that when $L=\Tr$ and $X=\xi_p$ we recover the regular exponential sum $S_{\mathbb{F}_q}(F)$.  Therefore, the study of regular exponential sums is embedded in the study of generating functions of the form (\ref{genfundef}). Thus, from now on, we consider the generating functions (\ref{genfundef}) instead of exponential sums of the form (\ref{expsumFq}).  Furthermore, in this article we use the term exponential sums to refer to both (\ref{expsumFq}) and (\ref{genfundef}).  

In \cite{cmsep}, closed formulas for exponential sums of type (\ref{expsumFq}) of elementary symmetric polynomials were found (extending the results of \cite{cai} to every finite field).  
There is a natural connection between the formulas presented in \cite{cmsep} and the value distribution of elementary symmetric polynomials over $\mathbb{F}_q$.  Part of the focus of this article is to explain 
such connection and to extend it to perturbations of elementary symmetric polynomials.  

Let $k$ be a natural number.  The elementary symmetric polynomial of degree $k$ in the variables $X_1,\ldots, X_n$ is denoted by $\boldsymbol{e}_k(X_1,\ldots, X_n)$.  Sometimes we use the more compact notation 
$\boldsymbol{e}_{n,k}$ to represent that polynomial, that is $\boldsymbol{e}_{n,k}$ also represents the $n$-variable elementary symmetric polynomial of degree $k$.
In this article, we prefer to use the notation $\boldsymbol{e}_{n,k}$ to represent the  $n$-variable elementary symmetric polynomial of degree $k$ when it has not been evaluated and the notation $\boldsymbol{e}_k({\bf x})$ when we want to stress that the elementary polynomial has been evaluated at ${\bf x}$.   

Let $F({\bf X}) \in \mathbb{F}_q[X_1,\ldots, X_j]$ ($j$ fixed). 
The polynomial $\boldsymbol{e}_{n,k}+F({\bf X})$ is called a {\it perturbation} of the $n$-variable elementary symmetric polynomial of degree $k$.   These  perturbations were introduced in \cite{cm2} for 
the binary case and are the main focus of \cite{cgm2,cm2}.  Perturbations break the symmetry of $\boldsymbol{e}_{n,k}$ and may reduce symmetry attacks in cryptographic implementations. 

The value of $\boldsymbol{e}_{k}(\bf x)$ is important when its exponential sums are studied.  Consider the set $A=\{0,a_1,\ldots,a_s\}$, where $a_j$'s are parameters.  Suppose that ${\bf x} \in A^n$ and that $a_j$ appears $m_j$ times in ${\bf x}$.  Following \cite{cmsep}, the value of $\boldsymbol{e}_k({\bf x})$
will be denoted by $\Lambda_{a_1,\ldots, a_s}(k,m_1,\ldots, m_s)$.  In the particular case when the set $A$ is the finite field $\mathbb{F}_q$ we use the notation $\Lambda_{\mathbb{F}_q^{\times}}(k,m_1,\ldots, m_{q-1})$.
A recursive definition for $\Lambda_{a_1,\ldots, a_s}(k,m_1,\ldots, m_s)$, which allows for fast evaluations of it, appears in \cite{cmsep}: 
\begin{eqnarray}
\Lambda_{a_1}(k,m)&=&a_1^k \binom{m}{k} \\\nonumber
\Lambda_{a_1,a_2,\ldots, a_{l+1}}(k,m_1,m_2,\ldots,m_{l+1})&=&
 \sum_{j=0}^{m_{l+1}}\binom{m_{l+1}}{j}a_{l+1}^j\Lambda_{a_1,\ldots,a_l}(k-j,m_1,m_2,\ldots,m_{l}).
\end{eqnarray}

As mentioned before, one of the main results of \cite{cmsep} are closed formulas for exponential sums of elementary symmetric polynomials over any finite field.  For convenience, we include their result next.  The result is written in terms of (\ref{genfundef}).


\begin{theorem}[\cite{cmsep}]
 \label{closedformsSq}
 Let $n$ and $k>1$ be positive integers, $p$ be a prime and $q=p^r$ with $r\geq 1$. Let $L:\mathbb{F}_q\to \mathbb{F}_q$ a linear function, $X$ an indeterminate  and $D=p^{\lfloor\log_p(k)\rfloor+1}$. Then,
 \begin{equation}
\label{closedformGF}
 \s_{\mathbb{F}_q,L}(\boldsymbol{e}_{n,k};X) = \sum_{j_1=0}^{D-1}\sum_{j_2=0}^{j_1}\ldots \sum_{j_{q-1}=0}^{j_{q-2}} 
 c_{j_1,\ldots,j_{q-1};L}(k;X)\left(1+\xi_D^{-j_1}+\cdots+\xi_D^{-j_{q-1}}\right)^n, 
\end{equation}
 where 
 \begin{equation*}
  c_{j_1,\ldots,j_{q-1};L}(k;X)=\frac{1}{D^{q-1}}\sum_{b_{q-1}=0}^{D-1}\ldots \sum_{b_1=0}^{D-1} X^{L\left(\Lambda_{\mathbb{F}_q^{\times}}\left(k,b_1,\ldots,b_{q-1}\right)\right)}
  \sum_{(j_1',\ldots,j_{q-1}')\in \Sym(j_1,\ldots, j_{q-1})}\xi_D^{j_1'b_{q-1}+\cdots+j_{q-1}' b_1},
 \end{equation*}
$\xi_{D}=\exp(2\pi i/D)$, and $\Sym(j_1,\ldots,j_{q-1})$ is the set of all rearrangements of $(j_1,\ldots,j_{q-1})$.
\end{theorem}
\begin{remark}
 Theorem \ref{closedformsSq} can be generalized without too much effort to linear combinations of elementary symmetric polynomials.  See \cite{cmsep} for more details.
\end{remark}

Theorem \ref{closedformsSq} is a generalization of the results presented in \cite{cai} for the binary field.  In \cite{cm1}, Castro and Medina used the closed formulas in \cite{cai} to calculate the asymptotic 
behavior of exponential sums of symmetric Boolean functions.  A similar result is now available in every finite field, that is,  Theorem \ref{closedformsSq} can be used to study the asymptotic behavior of 
exponential sums of the form $S_{\mathbb{F}_q}(\boldsymbol{e}_{n,k})$.  


The rest of the article is divided into three sections.  In the next one (Section 2) we study the asymptotic behavior of generating functions of the type (\ref{genfundef}) for elementary symmetric polynomials and their perturbations. One of the reasons to study such behavior is to explore the veracity of an open problem related to balancedness.  The results presented in Section 2 generalize the results presented in \cite{cm2} from the binary field to any finite field. In the third section, we study the distribution of the values of elementary symmetric polynomials over $\mathbb{F}_q$.  To be more precise, we study the probability that $\boldsymbol{e}_k({\bf x})$ returns $\beta \in \mathbb{F}_q$ when the entries of ${\bf x}$ are randomly selected from $\mathbb{F}_q$.  We also introduce the concept of asymptotically balanced symmetric polynomial and asymptotically balanced perturbation and show that a perturbation $\boldsymbol{e}_{n,k}+F({\bf X})$ is asymptotically balanced over $\mathbb{F}_p$ ($p$ prime) if and only if $\boldsymbol{e}_{n,k}$ is asymptotically balanced or $F({\bf X})$ is balanced over $\mathbb{F}_p$.  We also show that this statement is not true for finite fields in general and provide a way to construct counterexamples.  Finally, we finish the article with some concluding remarks.



\section{Asymptotic behavior of elementary symmetric polynomials and their perturbations}
\label{asympsec}

A function $F:\mathbb{F}_q\to\mathbb{F}_q$ is said to be {\it balanced} if its values are equally distributed.  That is, if $F$ takes each value of $\mathbb{F}_q$ exactly $q^{n-1}$ times.  Balancedness is important in some cryptographic implementations.  That is especially true when the characteristic of the field is 2.

There is an important conjecture proposed by Cusick, Li and St$\check{\mbox{a}}$nic$\check{\mbox{a}}$ about the balancedness of elementary symmetric polynomials over the binary field \cite{cusick2}.  Their conjecture states:

\begin{conjecture}[\cite{cusick2}]
\label{clsconj}
 There are no nonlinear balanced elementary symmetric Boolean functions except for degree $k=2^\ell$ and $2^{\ell+1}D-1$-variables, where $\ell,D$ are positive integers.
\end{conjecture}

\noindent
A generalized version of this conjecture for finite fields was presented in \cite{acgmr}.
\begin{conjecture}[\cite{acgmr}]
\label{clsconjFq}
 The only nonlinear balanced elementary symmetric polynomial over $\mathbb{F}_q$, $q=p^r$ are those with degree $k=p^\ell$ and $n=p^\ell D-1$ variables, where $\ell, D\in \mathbb{N}$, $D\not\equiv 1\mod p$.
\end{conjecture}
It is known that Conjecture \ref{clsconj} is true asymptotically \cite{cm1,fine,ggz}.  In particular, the argument presented in \cite{cm1} depends on a calculation of the asymptotic behavior of the exponential sum $S_{\mathbb{F}_2}(\boldsymbol{e}_{n,k})$.  Thus, to explore Conjecture \ref{clsconjFq}, it is natural to study the asymptotic behavior of $\s_{\mathbb{F}_q,L}(\boldsymbol{e}_{n,k};X)$.  Theorem \ref{closedformsSq} can be used to do that.

Consider the closed formula (\ref{closedformGF}) for $\s_{\mathbb{F}_q,L}(\boldsymbol{e}_{n,k},X)$.  Observe that $q$ is the biggest modulus of all complex numbers 
of the form $$1+\xi_D^{-j_1}+\cdots+\xi_D^{-j_{q-1}}, \text{ for }0\leq j_{q-1}\leq j_{q-2}\leq\cdots\leq j_1\leq D-1.$$  This maximum modulus is achieved if and only if $j_1=\cdots=j_{q-1} = 0$.  This implies that
\begin{eqnarray}
\label{asympbehk}
 \lim_{n\to\infty} \frac{1}{q^n} \s_{\mathbb{F}_q,L}(\boldsymbol{e}_{n,k},X) &=& c_{0,\ldots, 0;L}(k;X)\\ \nonumber
 &=&\frac{1}{D^{q-1}}\sum_{b_{q-1}=0}^{D-1}\cdots \sum_{b_1=0}^{D-1} X^{L\left(\Lambda_{\mathbb{F}_q^\times}^{(p)}\left(k,b_1,\ldots,b_{q-1}\right)\right)}.
\end{eqnarray}
Therefore, the asymptotic behavior of $\s_{\mathbb{F}_q,L}(\boldsymbol{e}_{n,k},X)$ is dominated by $c_{0,\ldots, 0}(k;X)\cdot q^n$.  We relabel $c_{0,\ldots,0;L}(k;X)$ as $c_{A,L}^{(q)}(k;X)$, i.e. 
\begin{equation}
\label{AC}
 c^{(q)}_{A,L}(k;X) = \frac{1}{D^{q-1}}\sum_{b_{q-1}=0}^{D-1}\cdots \sum_{b_1=0}^{D-1} X^{L\left(\Lambda_{\mathbb{F}_q^\times}\left(k,b_1,\ldots,b_{q-1}\right)\right)}.
\end{equation}
This is done in order to stress that (\ref{AC}) is the asymptotic coefficient for $\boldsymbol{e}_{n,k}$ over $\mathbb{F}_q$.  Observe that the value of $c^{(q)}_{A,L}(k;X)$ depends on knowing how many times, for each $\alpha\in \mathbb{F}_q$, $\Lambda_{\mathbb{F}_q^\times}^{(p)}\left(k,b_1,\ldots,b_{q-1}\right)=\alpha$ in a $q-1$-hypercube of side length $D$.  That is a very interesting combinatorial problem on its own.  For example, if we consider $q=3$, $k=27$ and $L(x)=x$, and color a point in the $81\times 81$ grid $\{(a,b)\,:\,0\leq a,b\leq 80\}$ blue if $\Lambda_{\mathbb{F}_3^\times}\left(27,a,b\right)=0$, red if $\Lambda_{\mathbb{F}_3^\times}\left(27,a,b\right)=1$ and green if $\Lambda_{\mathbb{F}_3^\times}\left(27,a,b\right)=2$, then we get the following picture.

\begin{figure}[h!]
\centering
\includegraphics[width=2.5in]{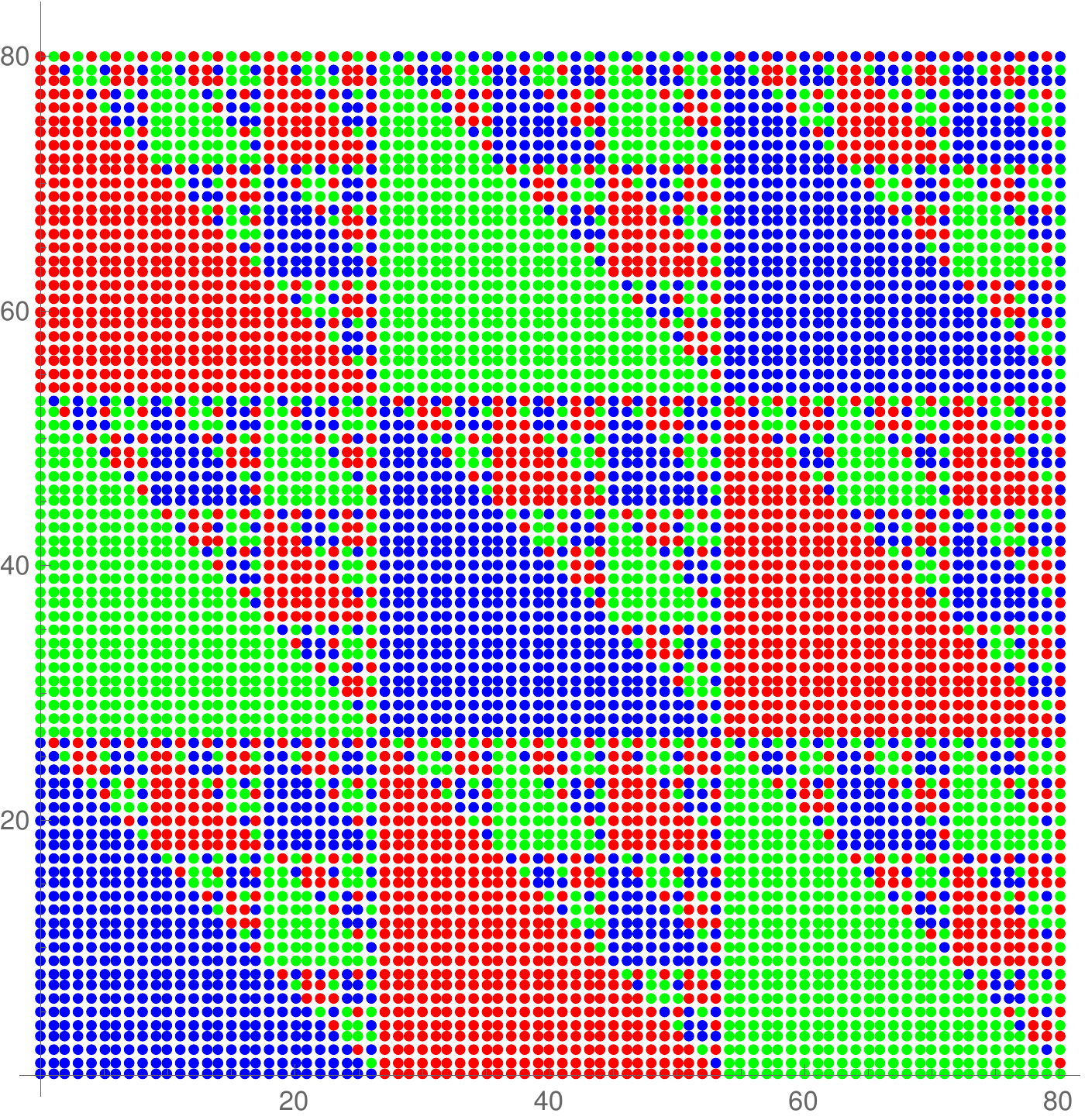}
\caption{Graphical representation of the values of $\Lambda_{\mathbb{F}_3^\times}\left(27,a,b\right)$ for $0\leq a,b\leq 80$.}
\label{badka2i3}
\end{figure}

The above argument can be easily extended to linear combinations of elementary symmetric polynomials.  Let $0<k_1<\cdots <k_s$ be integers and $\boldsymbol{k}=(k_1,\ldots, k_s)$.  
Let $\beta_1,\ldots, \beta_s \in \mathbb{F}_q$ and $\boldsymbol{\beta}=(\beta_1,\ldots,\beta_s)$.  Finally, let $D=p^{\lfloor\log_p(k_s)\rfloor+1}$.  Then,
 \begin{equation}
\label{asympbeh}
  \lim_{n\to\infty} \frac{1}{q^n} \s_{\mathbb{F}_q,L}\left(\sum_{t=1}^s \beta_t \boldsymbol{e}_{n,k_t};X\right) = c_{A,L}^{(q)}(\boldsymbol{k},\boldsymbol{\beta};X)
 \end{equation}
where
\begin{equation}
 c_{A,L}^{(q)}(\boldsymbol{k},\boldsymbol{\beta};X) = \frac{1}{D^{q-1}}\sum_{b_{q-1}=0}^{D-1}\cdots \sum_{b_{1}=0}^{D-1} X^{L\left(\sum_{t=0}^s \beta_t\Lambda_{\mathbb{F}_q^\times}(k_t,b_1,\ldots, b_{q-1})\right)}.
\end{equation}


Many of the results presented in \cite{cmsep} and the one presented so far can be extended to perturbations of elementary symmetric polynomials.   They follow from the fact that if
$F({\bf Y}) \in \mathbb{F}_q[Y_1,\cdots, Y_j]$ ($j$ fixed), then
\begin{equation}
\label{pertasreg}
\s_{\mathbb{F}_q,L}(\boldsymbol{e}_{n,k}+F({\bf Y});X) = \sum_{\boldsymbol{\beta} \in \mathbb{F}_q^j}X^{L(F(\boldsymbol{\beta}))} \s_{\mathbb{F}_q,L}\left(\sum_{m=0}^j \boldsymbol{e}_m(\boldsymbol{\beta})\boldsymbol{e}_{n-j,k-m};X\right),
\end{equation}
which is a consequence of the identity
\begin{equation}
\boldsymbol{e}_k(Y_1,\ldots, Y_n) = \sum_{m=0}^j \boldsymbol{e}_m(Y_1,\ldots, Y_j)\boldsymbol{e}_{k-m}(Y_{j+1},Y_{j+2},\ldots, Y_n).
\end{equation}
Observe that a corollary of (\ref{pertasreg}) is that exponential sums of perturbations of $\boldsymbol{e}_{n,k}$ have closed formulas similar to (\ref{closedformGF}).  Moreover, it is also true that 
they satisfy the linear recurrences presented in \cite{cmsep}.  This implies that a result similar to (\ref{asympbeh}) is expected.  The next two lemmas are going to be used to prove such claim.

\begin{lemma}
\label{lambdarel}
Consider the set $A=\{0,a_1,\ldots, a_s\}$ where $a_j$ are parameters.  Let $m_1,\ldots, m_s,l_1,\ldots, l_s$ be non-negative integers.  Suppose that $j\geq m_1+\cdots+m_s$.  Then,
\begin{equation}
\sum_{m=0}^j \Lambda_{a_1,\cdots, a_s}(m,m_1,\ldots,m_s)\Lambda_{a_1,\ldots, a_s}(k-m,l_1,\ldots,l_s) = \Lambda_{a_1,\ldots, a_s}(k,m_1+l_1,\ldots,m_s+l_s).
\end{equation}
\end{lemma}

\begin{proof}
This is a natural consequence of the equation
\begin{equation}
\boldsymbol{e}_k(X_1,\ldots, X_n) = \sum_{m=0}^j \boldsymbol{e}_m(X_1,\ldots, X_j)\boldsymbol{e}_{k-m}(X_{j+1},X_{j+2},\ldots, X_n)
\end{equation}
and the fact that if ${\bf x} \in A^n$ is such that $a_i$ appears $n_i$ times in ${\bf x}$, then $\boldsymbol{e}_k({\bf x}) =  \Lambda_{a_1,\ldots, a_s}(m,n_1,\ldots,n_s)$.  Observe that $j$ must be 
bigger than or equal to $m_1+\cdots+m_s$ in order to have the necessary amount of variables to support $m_1+\cdots+m_s$ values.
\end{proof}

\begin{lemma}
Let $k$ be a positive integer.  Suppose that $\beta_1,\ldots, \beta_j \in \mathbb{F}_q$ where $q=p^r$ with $p$ prime and that $L:\mathbb{F}_q\to \mathbb{F}_q$ is a linear function.  Then,
\begin{equation}
\lim_{n\to \infty} \frac{1}{q^{n}} \s_{\mathbb{F}_q,L}\left(\sum_{m=0}^j \boldsymbol{e}_m(\beta_1,\ldots, \beta_j)\boldsymbol{e}_{n,k-m};X\right) = 
\lim_{n\to \infty}\frac{1}{q^n} \s_{\mathbb{F}_q,L}(\boldsymbol{e}_{n,k};X) = c_{A,L}^{(q)}(k;X).
\end{equation}
\end{lemma}

\begin{proof}
Let $\boldsymbol{\beta}=(\beta_1,\ldots,\beta_j)$ and $D=p^{\lfloor\log_p(k)\rfloor+1}$.  Recall that
\begin{eqnarray}
 c^{(q)}_{A,L}(k;X)&=&\lim_{n\to \infty}\frac{1}{q^n} \s_{\mathbb{F}_q,L}(\boldsymbol{e}_{n,k};X)\\\nonumber
c_{A,L}^{(q)}(k,\boldsymbol{\beta};X)&=& \lim_{n\to \infty} \frac{1}{q^{n}} \s_{\mathbb{F}_q,L}\left(\sum_{m=0}^j \boldsymbol{e}_m(\beta_1,\ldots, \beta_j)\boldsymbol{e}_{n,k-m};X\right)
\end{eqnarray}
where
\begin{eqnarray}
c^{(q)}_{A,L}(k;X) &=& \frac{1}{D^{q-1}}\sum_{b_{q-1}=0}^{D-1}\cdots \sum_{b_{1}=0}^{D-1} X^{L(\Lambda_{\mathbb{F}_q^\times}(k,b_1,\ldots, b_{q-1}))}\\\nonumber
c_{A,L}^{(q)}(k,\boldsymbol{\beta};X) &=& \frac{1}{D^{q-1}}\sum_{b_{q-1}=0}^{D-1}\cdots \sum_{b_{1}=0}^{D-1} X^{L\left(\sum_{m=0}^j \boldsymbol{e}_m(\beta_1,\ldots,\beta_j)\Lambda_{\mathbb{F}_q^\times}(k-m,b_1,\ldots, b_{q-1})\right)}.
\end{eqnarray}
Let us work with $c_{A,L}^{(q)}(k,\boldsymbol{\beta};X)$.  Suppose that $\alpha_t$ appears $b_t'$ times in the entries of the vector $(\beta_1,\ldots, \beta_j)$.  Then, Lemma \ref{lambdarel} implies
\begin{eqnarray}
c_{A,L}^{(q)}(k,\boldsymbol{\beta};X) &=& \frac{1}{D^{q-1}}\sum_{b_{q-1}=0}^{D-1}\cdots \sum_{b_{1}=0}^{D-1} X^{L\left(\sum_{m=0}^j \boldsymbol{e}_m(\beta_1,\ldots,\beta_j)\Lambda_{\mathbb{F}_q^\times}(k-m,b_1,\ldots, b_{q-1})\right)}\\\nonumber
&=&\frac{1}{D^{q-1}}\sum_{b_{q-1}=0}^{D-1}\cdots \sum_{b_{1}=0}^{D-1} X^{L(\Lambda_{\mathbb{F}_q^\times}(k,b_1+b_1',\ldots, b_{q-1}+b_{q-1}'))}.
\end{eqnarray}
However, we know that $\Lambda_{\mathbb{F}_q^\times}(k,m_1,\ldots, m_{q-1})$ is periodic mod $p$ in each of the entries $m_1,\ldots, m_{q-1}$ with period length $D$ (see \cite{cmsep}).  Since each of the 
variables $b_t$ runs a full period, i.e. from 0 to $D-1$, then 
\begin{eqnarray}
c_{A,L}^{(q)}(k,\boldsymbol{\beta};X) &=&\frac{1}{D^{q-1}}\sum_{b_{q-1}=0}^{D-1}\cdots \sum_{b_{1}=0}^{D-1} X^{L(\Lambda_{\mathbb{F}_q^\times}(k,b_1+b_1',\ldots, b_{q-1}+b_{q-1}'))}\\\nonumber
&=& \frac{1}{D^{q-1}}\sum_{d_{q-1}=0}^{D-1}\cdots \sum_{d_{1}=0}^{D-1} X^{L(\Lambda_{\mathbb{F}_q^\times}(k,d_1,\ldots, d_{q-1}))}\\\nonumber
&=& c^{(q)}_{A,L}(k;X).
\end{eqnarray}
This concludes the proof.
\end{proof}

Next is the generalization of (\ref{asympbeh}).  As before, it is stated for perturbations of elementary symmetric polynomials, but the same holds true for perturbations of linear combinations of them.

\begin{theorem}
\label{asympcoeff}
Let $k>1$ be an integer, $p$ a prime and $q=p^r$ with $r\geq 1$.  Suppose that $F({\bf Y})$ is a polynomial in the variables $Y_1,\ldots, Y_j$ ($j$ fixed) with coefficients from $\mathbb{F}_q$ and that 
$L:\mathbb{F}_q\to \mathbb{F}_q$ is a linear function.  Then,
\begin{equation}
\lim_{n\to \infty}\frac{1}{q^n} \s_{\mathbb{F}_q,L}(\boldsymbol{e}_{n,k}+F({\bf Y});X) = \frac{1}{q^j}c^{(q)}_{A,L}(k;X)\s_{\mathbb{F}_q,L}(F;X).
\end{equation}
\end{theorem}

\begin{proof}
By Theorem \ref{pertasreg} we know that
\begin{equation}
\s_{\mathbb{F}_q,L}(\boldsymbol{e}_{n,k}+F({\bf Y});X) = \sum_{\boldsymbol{\beta}\in \mathbb{F}_q^j}X^{L(F(\boldsymbol{\beta}))}\s_{\mathbb{F}_q,L}\left(\sum_{m=0}^j \boldsymbol{e}_m(\boldsymbol{\beta})
\boldsymbol{e}_{n-j,k-m};X\right).
\end{equation}
Therefore,
\begin{eqnarray}\nonumber
\lim_{n\to \infty} \frac{1}{q^n}\s_{\mathbb{F}_q}(\boldsymbol{e}_{n,k}+F({\bf Y});X) &=& \sum_{\boldsymbol{\beta}\in \mathbb{F}_q^j}X^{L(F(\boldsymbol{\beta}))} 
\left( \lim_{n\to \infty} \frac{1}{q^n} \s_{\mathbb{F}_q,L}\left(\sum_{m=0}^j \boldsymbol{e}_m(\boldsymbol{\beta})\boldsymbol{e}_{n-j,k-m};X\right)\right)\\
&=& \frac{1}{q^j}\sum_{\boldsymbol{\beta}\in \mathbb{F}_q^j}X^{L(F(\boldsymbol{\beta}))} \left( \lim_{n\to \infty} \frac{1}{q^{n-j}} \s_{\mathbb{F}_q,L}\left(\sum_{m=0}^j \boldsymbol{e}_m
(\boldsymbol{\beta})\boldsymbol{e}_{n-j,k-m};X\right)\right)\\\nonumber
&=& \frac{1}{q^j}\sum_{\boldsymbol{\beta}\in \mathbb{F}_q^j}X^{L(F(\boldsymbol{\beta}))} c^{(q)}_{A,L}(k;X)\\\nonumber
&=& \frac{1}{q^j}c^{(q)}_{A,L}(k;X)\s_{\mathbb{F}_q,L}(F;X).
\end{eqnarray}
This concludes the proof.
\end{proof}

Theorem \ref{asympcoeff} is also a generalization of the main theorem of \cite[Th. 4.4]{cm2}.  In fact, the discussion so far about perturbations of elementary symmetric polynomials generalizes most of the results presented in \cite{cm2} for the binary field.  In the next section we show how the results presented in this section can be used to study the distribution of the values of $\boldsymbol{e}_{n,k}$ (and its perturbations) in finite fields.
\section{Distribution of the values of elementary symmetric polynomials and their perturbations over $\mathbb{F}_q$}
\label{distributionSec}

The generating function $\s_{\mathbb{F}_q;L}(F;X)$ can be written as
\begin{equation}
\s_{\mathbb{F}_q,L}(F;X) = \sum_{\beta\in \mathbb{F}_q}N_{\mathbb{F}_q,L}(F;\beta)X^\beta,
\end{equation}
where 
\begin{equation}
 N_{\mathbb{F}_q,L}(F;\beta) = \left|\{{\bf x} \in \mathbb{F}_q^n \,:\, L(F({\bf x})) = \beta\}\right|, \,\,\,\beta \in \mathbb{F}_q.
 \footnote{Observe that since $\{\s_{\mathbb{F}_q,L}(\boldsymbol{e}_{n,k};X)\}_{n\in \mathbb{N}}$ satisfies the linear recurrence with integer coefficients provided in \cite[Th. 5.7]{cmsep}, then 
 $\{N_{\mathbb{F}_q,L}(\boldsymbol{e}_{n,k};\beta)\}_{n\in \mathbb{N}}$ also satisfies such recurrence.  Also, $N_{\mathbb{F}_q,L}(\boldsymbol{e}_{n,k};\beta)$ has a closed formula similar to the one in Theorem \ref{closedformsSq}.}
\end{equation}
This implies that
\begin{equation}
 \frac{1}{q^n}\s_{\mathbb{F}_q,L}(\boldsymbol{e}_{n,k};X) = \sum_{\beta\in \mathbb{F}_q} \mathbbm{p}^{(q)}_{n,k}(\beta;L) X^\beta
\end{equation}
where $\mathbbm{p}^{(q)}_{n,k}(\beta;L) = N_{\mathbb{F}_q,L}(\boldsymbol{e}_{n,k};\beta)/q^n$ is the probability that $L(\boldsymbol{e}_k({\bf x}))$ returns the value $\beta\in\mathbb{F}_q$ when 
${\bf x}$ is randomly selected from $\mathbb{F}_q^n$.

Equation (\ref{asympbehk}) states that
\begin{equation}
 \lim_{n\to \infty} \frac{1}{q^n}\s_{\mathbb{F}_q,L}(\boldsymbol{e}_{n,k};X) = c^{(q)}_{A,L}(k;X).
\end{equation}
Expressing $c^{(q)}_{A,L}(k;X)$ as
\begin{equation}
 c^{(q)}_{A,L}(k;X)= \sum_{\beta\in \mathbb{F}_q} a_\beta X^\beta,
\end{equation}
we see that 
\begin{equation}
\label{nicelim}
 a_\beta = \lim_{n\to \infty} \mathbbm{p}^{(q)}_{n,k}(\beta;L) = \lim_{n\to \infty} \frac{N_{\mathbb{F}_q,L}(\boldsymbol{e}_{n,k};\beta)}{q^n}:=\mathbbm{p}^{(q)}_{k}(\beta;L).
\end{equation}
The limit in (\ref{nicelim}) exists and we call $\mathbbm{p}^{(q)}_{k}(\beta;L)$ the {\em probability at infinity} that $L(\boldsymbol{e}_k({\bf x}))$ returns the value $\beta \in \mathbb{F}_q$ when ${\bf x}$ is randomly selected with entries from $\mathbb{F}_q$.  Clearly, if $\varepsilon>0$, then for all $n$ big enough,
\begin{equation}
\left|\mathbbm{p}^{(q)}_{n,k}(\beta;L)-\mathbbm{p}^{(q)}_{k}(\beta;L)\right|<\varepsilon\,\,\, \text{ for all }\beta\in \mathbb{F}_q.
\end{equation}
Also,
\begin{equation}
 N_{\mathbb{F}_q,L}(\boldsymbol{e}_{n,k};\beta) \sim \mathbbm{p}^{(q)}_{k}(\beta;L) \cdot q^n.
\end{equation}

Observe that under this setting $c_{A,L}^{(q)}(k;X)$ is the probability generating function for $\mathbbm{p}^{(q)}_{k}(\beta;L)$. Therefore, the study of the distribution of the values of 
$L(\boldsymbol{e}_{k}({\bf X}))$ in $\mathbb{F}_q$ is equivalent to the study of $c_{A,L}^{(q)}(k;X)$.  We express the probability generating function for $\mathbbm{p}^{(q)}_{n,k}(\beta;L)$ as
$\G_{n,k}^{(q)}(L;X)$, that is
\begin{equation}
\G_{n,k}^{(q)}(L;X) = \sum_{\beta \in \mathbb{F}_q} \mathbbm{p}^{(q)}_{n,k}(\beta;L) X^{\beta}.
\end{equation}
We also relabel $c_{A,L}^{(q)}(k;X)$ as $\G_{k}^{(q)}(L;X)$ in an attempt to make the fact that $c_{A,L}^{(q)}(k;X)$ is the probability generating function of $\mathbbm{p}^{(q)}_{k}(\beta;L)$ more clear.

The next theorem summarizes the discussion so far.  Again, it is stated for elementary symmetric polynomials, but it can be extended to linear combinations of them.

\begin{theorem}
\label{genfun}
Let $p$ be a prime, $q=p^r$ where $r\geq 1$ and $L:\mathbb{F}_q\to\mathbb{F}_q$ be a linear function.  Suppose that $k>1$ is an integer and $D=p^{\lfloor\log_p(k)\rfloor+1}$. 
Then,
\begin{equation}
  \lim_{n\to \infty} \G_{n,k}^{(q)}(L;X)=\frac{1}{D^{q-1}}\sum_{b_{q-1}=0}^{D-1}\cdots \sum_{b_1=0}^{D-1} X^{L\left(\Lambda_{\mathbb{F}_q^\times}(k,b_1,\ldots,b_{q-1})\right)}=\G_k^{(q)}(L;X).
 \end{equation}
\end{theorem}

The study of perturbations of the form $\boldsymbol{e}_{n,k}+F({\bf X})$ follows in an analogous way.  We use the notation $\G_{n,k;F}^{(q)}(L;X)$ to represent 
\begin{equation}
\G_{n,k;F}^{(q)}(L;X) = \sum_{\beta \in \mathbb{F}_q} \mathbbm{p}^{(q)}_{n,k;F}(\beta;L) X^{\beta},
\end{equation}
with $\mathbbm{p}^{(q)}_{n,k;F}(\beta;L)$ defined in the natural way.  As in the previous discussion, the limit
\begin{equation}
 \lim_{n\to \infty} \mathbbm{p}^{(q)}_{n,k;F}(\beta;L)
\end{equation}
exists.  The value of the limit is denoted by $\mathbbm{p}^{(q)}_{k;F}(\beta;L)$ and we use $\G_{k;F}^{(q)}(L;X)$ to represent the probability generating function of $\mathbbm{p}^{(q)}_{k;F}(\beta;L)$.  Observe that the conclusion of Theorem \ref{asympcoeff} can be re-stated as
\begin{eqnarray}
 \lim_{n\to \infty} \G_{n,k;F}^{(q)}(L;X)&=& \G_{k;F}^{(q)}(L;X)\\\nonumber
 &=& \frac{1}{q^j}\G_{k}^{(q)}(L;X) \s_{\mathbb{F}_q,L}(F;X).
\end{eqnarray}

\begin{remark}
When $L(X)=X$, we drop the ``$L$" in the notation of our functions.  For example, we write $\mathbbm{p}^{(q)}_{n,k}(\beta)$ instead of $\mathbbm{p}^{(q)}_{n,k}(\beta;L)$ or $\G_{k}^{(q)}(X)$ instead of $\G_{k}^{(q)}(L;X)$.
\end{remark}

\begin{example}
Consider the polynomial $\boldsymbol{e}_{5}({\bf X})$ over $\mathbb{F}_4=\mathbb{F}_2(\alpha)$ with $\alpha^2+\alpha+1=0$.  In this case,

\begin{equation}
\G_{5}^{(4)}(X) = \frac{11}{32}+\frac{7}{32}X+\frac{7}{32}X^{\alpha }+\frac{7}{32}X^{\alpha +1}.
\end{equation}
This implies that the probability at infinity that $\boldsymbol{e}_{5}({\bf x})$ returns 0 is $11/32$ and the probabilities that it returns $1,\alpha$ and $\alpha+1$ are all $7/32$.

Let $F({\bf X})=X_1 X_2+X_1 X_3 X_2+X_3 X_2+X_1 X_3$ and consider the perturbation polynomial $\boldsymbol{e}_{n,5}+F({\bf X})$.  Theorem \ref{asympcoeff} implies that

\begin{eqnarray}
\G_{5;F}^{(4)}(X) &=& \frac{1}{4^3}\G_{5}^{(4)}(X)\s_{\mathbb{F}_4}(F;X)\\\nonumber
&=& \frac{1}{64}\left(\frac{11}{32}+\frac{7}{32}X+\frac{7}{32}X^{\alpha }+\frac{7}{32}X^{\alpha +1}\right)\left(17+21X+13X^{\alpha}+13X^{\alpha+1}\right)\\\nonumber
&=&\frac{187}{2048}+\frac{175}{1024}X+\frac{147}{2048}X^2+\frac{131}{1024}X^{\alpha}+\frac{91}{2048}X^{2\alpha}+\\\nonumber
& &\frac{125}{512}X^{\alpha+1}+\frac{119}{1024}X^{\alpha+2}+\frac{91}{1024}X^{2\alpha+1}+\frac{91}{2048}X^{2\alpha+2}\\\nonumber
&=& \frac{129}{512}+\frac{133}{512}X+\frac{125}{512}X^{\alpha}+\frac{125}{512}X^{\alpha+1},
\end{eqnarray}
where the last equation comes from the fact that we are working on $\mathbb{F}_4=\mathbb{F}_2(\alpha)$.  Observe that this implies that the probability at infinity that $\boldsymbol{e}_{5}({\bf x})+F({\bf x})$ returns 0 is $129/512$, the probability it returns 1 is $133/512$ and the probabilities that it returns $\alpha$ and $\alpha+1$ are $125/512$ each.
\end{example}

\begin{example}
Consider now the polynomial $\boldsymbol{e}_{4}({\bf X})$ over $\mathbb{F}_9=\mathbb{F}_3(\alpha)$, where $\alpha^2+1=0$.  In this case,
\begin{equation}
\label{overF9}
\G_{4}^{(9)}(X) = \sum_{\beta \in \mathbb{F}_3}\frac{29}{243}X^\beta+\sum_{\beta \in \mathbb{F}_9\setminus \mathbb{F}_3}\frac{26}{243}X^\beta.
\end{equation}
Consider now the perturbation $\boldsymbol{e}_{n,4}+F({\bf X})$ where $F({\bf X})=X_1 X_2 X_3+X_1 X_2+X_3$.  Observe that
\begin{eqnarray}
 \G_{4;F}^{(9)}(X)&=& \frac{1}{9^3}\G_{4}^{(9)}(X)\s_{\mathbb{F}_9}(F;X)\\\nonumber
 &=& \frac{1}{729}\left(\sum_{\beta \in \mathbb{F}_3}\frac{29}{243}X^\beta+\sum_{\beta \in \mathbb{F}_9\setminus \mathbb{F}_3}\frac{26}{243}X^\beta\right)\left(145X^2+\sum_{\beta \in\mathbb{F}_9\setminus\{2\}} 73 X^\beta\right)\\\nonumber
 &=& \sum_{\beta \in \mathbb{F}_3}\frac{2203}{19683}X^\beta+\sum_{\beta \in \mathbb{F}_9\setminus \mathbb{F}_3}\frac{2179}{19683}X^\beta.
\end{eqnarray}

\end{example}

One of the first persons to study  (if not the first one) the asymptotic distribution of the values of elementary symmetric polynomials over finite fields of {\it prime order} was N. J. Fine \cite{fine}.  He proved that $\mathbbm{p}^{(p)}_{k}(t)$ exists for every prime $p$ and calculated the distribution of $\mathbbm{p}^{(p)}_{k}(t)$ for $p=2,3$.  He also established that for $p$ equal to 2 or 3 (highlighted by Aberth \cite{aberth}),
\begin{enumerate}
\item $\mathbbm{p}^{(p)}_{k}(0)\geq 1/p$,
\item $\mathbbm{p}^{(p)}_{k}(0) = 1/p$ only if  $k= d \cdot p^l$ where $1\leq d \leq p-1$,
\item $\mathbbm{p}^{(p)}_{k}(t) = 1/p$ if  $k= d \cdot p^l$ where $1\leq d \leq p-1$,
\item $\mathbbm{p}^{(p)}_{k p}(t) = \mathbbm{p}^{(p)}_{k}(t)$,
\item $\mathbbm{p}^{(p)}_{k}(0) \geq \mathbbm{p}^{(p)}_{k}(t)$ with equality only for $k= d \cdot p^l$ where $1\leq d \leq p-1$.
\end{enumerate}
Fine also proved (3) for all $p$, which implies that the proof of the generalization of the conjecture of Cusick, Li and St$\check{\mbox{a}}$nic$\check{\mbox{a}}$ presented in \cite{acgmr} is expected to be 
much harder than the binary counterpart.  In particular, when $p>2$, the approach presented in \cite{cm1} will fail to prove the conjecture asymptotically when $k=d \cdot p^l$ and $1< d \leq p-1$.

Fine proposed as problems the veracity of the other properties for general $p$.  O. Aberth \cite{aberth} disproved (2) and (5) by showing that $\mathbbm{p}^{(5)}_6(0)=1/5$ and $\mathbbm{p}^{(5)}_6(2)=26/125$.  He also showed that $\mathbbm{p}^{(5)}_{30}(0)=15749/78125>1/5=\mathbbm{p}^{(5)}_6(0)$ and therefore (4) is also false.  In \cite{jdsmith}, J. D. Smith generalized Aberth's example and showed that if $p>3$ is prime, then
\begin{equation}
\mathbbm{p}^{(p)}_{p+1}(t) = \begin{cases}
 \frac{1}{p}, & t=0 \\
 \frac{1}{p} +   \leg{2t}p \frac{1}{p^{\mu}}, & t\neq 0,
\end{cases}
\end{equation}
where $\mu=(p+1)/2$ and $\leg{a}p$ represents the Legendre symbol.  Smith's general formula for $\mathbbm{p}^{(p)}_k(t)$ as a multisum coincides with our formula in Theorem \ref{genfun} for $q=p$ and
$L(X)=X$.

As mentioned at the beginning of Section \ref{asympsec}, one of the reasons the asymptotic behavior of exponential sums of symmetric polynomials was calculated over the binary field was to provide an asymptotic proof of Conjecture \ref{clsconj} (see \cite{cm1}).  The concept of {\it asymptotically balanced symmetric Boolean function} was introduced in \cite{cm1} to mean that $c_A^{(2)}(k;-1)=0$. Conjecture \ref{clsconj} was proved asymptotically in \cite{cm1} by showing that $\boldsymbol{e}_{n,k}$ is asymptotically balanced if and only if $k$ is a power of two.  Observe that if a polynomial is not asymptotically balanced, then we know that it is not balanced for a sufficiently large number of variables.  Thus, asymptotically balanced polynomials are good candidates for regular balancedness.  The concept of asymptotically balanced polynomials was extended to perturbations of elementary symmetric polynomials in \cite{cm2}.

A natural generalization for the concept of asymptotically balanced symmetric polynomial over $\mathbb{F}_q$ is to say that a polynomial $\boldsymbol{e}_{n,k}$ is {\it asymptotically balanced} if and only if 
\begin{equation}
 \mathbbm{p}_k^{(q)}(\beta)=\frac{1}{q},\,\,\, \text{ for every }\beta\in\mathbb{F}_q.
\end{equation}
The concept can also be extended to perturbations in the only natural way, that is, by saying that a perturbation $\boldsymbol{e}_{n,k}+F({\bf X})$ is asymptotically balanced if and only if 
\begin{equation}
 \mathbbm{p}_{k;F}^{(q)}(\beta)=\frac{1}{q},\,\,\, \text{ for every }\beta\in\mathbb{F}_q.
\end{equation}
Observe that Fine \cite{fine} proved that $\boldsymbol{e}_{n,k}$ is asymptotically balanced over the prime field $\mathbb{F}_p$ when $k= d \cdot p^l$ where $1\leq d \leq p-1$.  In \cite[Th. 2]{acgmr}, it was proved that if $q=p^r$, then $\boldsymbol{e}_{n,p^\ell}$ is asymptotically balanced over $\mathbb{F}_q$ for every $\ell$.

One of the main goals in \cite{cm2} was to identify when a particular pertubation is asymptotically balanced over $\mathbb{F}_2$.  It was showed \cite[Cor. 4.5]{cm2} that a perturbation $\boldsymbol{e}_{n,k}+F({\bf X})$ is asymptotically balanced over $\mathbb{F}_2$ if and only if $\boldsymbol{e}_{n,k}$ is asymptotically balanced or $F({\bf X})$ is a balanced function.  The same result holds true over any prime field, but it is not necessarily true over finite fields in general.

\begin{proposition}
\label{PropAsympBalanced}
 Let $p$ be a prime.  Suppose that $F({\bf X}) \in \mathbb{F}_p[X_1,\cdots, X_j]$ ($j$ fixed).  Then, 
 $$\mathbbm{p}^{(p)}_{k;F}(t) = \frac{1}{p}, \,\, \text{ for every }t \in \mathbb{F}_p$$
 if and only if $\mathbbm{p}^{(p)}_{k}(t)=1/p$ for every $t\in \mathbb{F}_p$ or $\s_{\mathbb{F}_p}(F;X) =\sum_{t \in \mathbb{F}_p} p^{j-1} X^t$.  In other words,  $\boldsymbol{e}_{n,k}+F({\bf X})$ is asymptotically balanced if and only if $\boldsymbol{e}_{n,k}$ is asymptotically balanced or $F({\bf X)}$ is balanced.
\end{proposition}

\begin{proof}
Theorem \ref{asympcoeff} implies that 
\begin{equation}
\label{genfunprod}
 \G_{k;F}^{(p)}(X)= \frac{1}{p^j}\G_{k}^{(p)}(X) \s_{\mathbb{F}_p}(F;X).
\end{equation}
Suppose first that $\mathbbm{p}^{(p)}_{k}(t)=1/p$ for every $t\in \mathbb{F}_p$ or $\s_{\mathbb{F}_p}(F;X) =\sum_{t \in \mathbb{F}_p} p^{j-1} X^t$. Then, the equation
\begin{equation}
\label{polytrick}
 \sum_{\beta \in \mathbb{F}_q}a_{\beta} X^{\beta}\sum_{\beta \in \mathbb{F}_q} X^{\beta} = 
 \sum_{\beta \in \mathbb{F}_q}\left(\sum_{\gamma \in \mathbb{F}_q}a_{\gamma}\right) X^{\beta},
\end{equation}
which is true for any finite field $\mathbb{F}_q$, together with (\ref{genfunprod}) imply that the coefficients of $\G_{k;F}^{(p)}(X)$ are all equal.  But that can only be true if
$$\mathbbm{p}^{(p)}_{k;F}(t) = \frac{1}{p}, \,\, \text{ for every }t \in \mathbb{F}_p.$$

To prove the other direction, let $X=\xi_p=\exp(2\pi i/p)$.  That transforms (\ref{genfunprod}) into 
\begin{equation}
 \G_{k;F}^{(p)}(\xi_p)= \frac{1}{p^j}\G_{k}^{(p)}(\xi_p) S_{\mathbb{F}_p}(F),
\end{equation}
where $S_{\mathbb{F}_p}(F)$ is the regular exponential sum of $F$ (a complex number).  If it is true that $\mathbbm{p}^{(p)}_{k;F}(t) = 1/p$ for every $t\in \mathbb{F}_p$, then 
$$\G_{k;F}^{(p)}(\xi_p)=\sum_{t\in \mathbb{F}_p} \frac{1}{p}\, \xi_p^t =0.$$
But then 
$$\frac{1}{p^j}\G_{k}^{(p)}(\xi_p) S_{\mathbb{F}_p}(F) = 0,$$
and so $\G_{k}^{(p)}(\xi_p)=0$ or $S_{\mathbb{F}_p}(F) = 0$.  If the latter is true, then $F({\bf X})$ is balanced over $\mathbb{F}_p$.  If $\G_{k}^{(p)}(\xi_p)=0$, then the minimal polynomial of $\xi_p$, i.e. $\Phi_p(X)=1+X+X^2+\cdots+X^{p-1}$, divides the polynomial $\G_{k}^{(p)}(X)$.  Since both polynomials are of the same degree, then $\G_{k}^{(p)}(X)$ is a constant multiple of $\Phi_p(X)$.  We conclude that $\mathbbm{p}^{(p)}_{k}(t)=1/p$ for every $t\in \mathbb{F}_p$, i.e. $\boldsymbol{e}_{n,k}$ is asymptotically balanced.  This concludes the proof.
\end{proof}

Proposition \ref{PropAsympBalanced} is not true for $\mathbb{F}_q$ when $q$ is not prime.  The sufficient part still holds and is a consequence of equation (\ref{polytrick}), but the necessary part is not true in general.  Next we present a method to construct counterexamples of Proposition \ref{PropAsympBalanced} over $\mathbb{F}_q$.  

\subsection{A construction for counterexamples over $\mathbb{F}_q$}

Let $q=p^r$ with $r>1$.  We want to find an elementary symmetric polynomial $\boldsymbol{e}_{n,k}$ and a polynomial $F({\bf X}) \in \mathbb{F}_q[X_1,\ldots,X_j]$, such that $\boldsymbol{e}_{n,k}$ is not asymptotically balanced over $\mathbb{F}_q$ and $F({\bf X})$ is not balanced over $\mathbb{F}_q$, but $\boldsymbol{e}_{n,k}+F({\bf X})$ is asymptotically balanced over $\mathbb{F}_q$.  

Suppose that $\boldsymbol{e}_{n,k}$ was selected such that it is not asymptotically balanced over $\mathbb{F}_q$.  Recall that 
\begin{equation}
\label{geneq}
 \G_{k;F}^{(q)}(X)= \frac{1}{q^j}\G_{k}^{(q)}(X) \s_{\mathbb{F}_q}(F;X).
\end{equation} 
Suppose that 
\begin{equation}
 \G_{k}^{(q)}(X)= \sum_{\beta\in\mathbb{F}_q} a_{\beta} X^{\beta} \text{ and }\frac{1}{q^j}\s_{\mathbb{F}_q}(F;X)= \sum_{\beta\in\mathbb{F}_q} b_{\beta} X^{\beta},
\end{equation}
where $0< a_\beta,b_\beta < 1$ and $\sum_{\beta\in\mathbb{F}_q}a_\beta=\sum_{\beta\in\mathbb{F}_q}b_\beta=1$ and that 
\begin{equation}
 \G_{k;F}^{(q)}(X)=\sum_{\beta\in\mathbb{F}_q} \frac{1}{q} X^\beta.
\end{equation}
Observe that, by assumption on $\boldsymbol{e}_{n,k}$, not all $a_{\beta}$'s are equal.  Equation (\ref{geneq}) can now be expressed as
\begin{equation}
 \sum_{\beta\in\mathbb{F}_q} \frac{1}{q} X^\beta = \sum_{\beta\in\mathbb{F}_q}\left(\sum_{\gamma\in\mathbb{F}_q}a_{\beta-\gamma}b_{\gamma}\right)X^\beta,
\end{equation}
which can be written in matrix form as
\begin{equation}
\label{matrixprob}
 \frac{1}{q}\boldsymbol{1} = A_{q,k}\cdot \boldsymbol{b},
\end{equation}
where $\boldsymbol{1}$ and $\boldsymbol{b}$ are the column vectors whose entries are all 1's and all the $b_\beta$'s (resp.), and $A_{q,k}$ is the $q\times q$ matrix $A_{q,k}=(a_{\beta-\gamma})_{\beta,\gamma}$.  The problem now is to verify if a solution to (\ref{matrixprob}) with $\boldsymbol{b}\neq (1/q)\boldsymbol{1}$ is possible.

Observe that $A_{q,k}$ is a doubly stochastic matrix.  That means that $(1/q)\boldsymbol{1}$ is an eigenvector (corresponding to the eigenvalue $\lambda=1$).  It also implies that (see \cite{ibe})
\begin{equation}
 \lim_{N\to \infty} A_{q,k}^N =\frac{1}{q}J_{q}:= \frac{1}{q} \left(
\begin{array}{cccc}
 1 & 1 & \cdots & 1 \\
 1 & 1 & \cdots & 1 \\
 \vdots & \vdots & \ddots & \vdots \\
 1 & 1 & \cdots & 1 \\
\end{array}
\right).
\end{equation}
Suppose that $A_{q,k}$ also happens to be singular.  Let $\boldsymbol{v}$ a non-trivial vector in the null space of $A_{q,k}$.  Then $A_{q,k}^N\boldsymbol{v}=\boldsymbol{0}$ for every $N$ and so
\begin{equation}
 \boldsymbol{0} = \lim_{N\to \infty} A_{q,k}^N \boldsymbol{v} = \frac{1}{q} J_q \boldsymbol{v},
\end{equation}
which implies that $v_1+\cdots+v_q=0$.  Now choose $\varepsilon>0$ small enough such that all entries of
\begin{equation}
 \frac{1}{q}\boldsymbol{1}+\varepsilon \boldsymbol{v} =  \left(\begin{array}{c}
 1/q + \varepsilon v_1  \\
 1/q + \varepsilon v_2\\
 \vdots\\
 1/q + \varepsilon v_q
\end{array}
\right)
\end{equation}
are positive.  Observe that 
$$\sum_{j=1}^q \left(\frac{1}{q}+\varepsilon v_j\right) = 1 + \varepsilon \sum_{j=1}^q v_j = 1,$$
which means that $(1/q)\boldsymbol{1}+\varepsilon \boldsymbol{v}$ is a probability vector different from $(1/q)\boldsymbol{1}$ that satisfies 
\begin{equation}
 A_{q,k} \left(\frac{1}{q}\boldsymbol{1}+\varepsilon \boldsymbol{v}\right) = \frac{1}{q}A_{q,k}\boldsymbol{1}+\varepsilon A_{q,k}\boldsymbol{v} = \frac{1}{q}\boldsymbol{1}+\boldsymbol{0}= \frac{1}{q}\boldsymbol{1}.
\end{equation}
In other words, $(1/q)\boldsymbol{1}+\varepsilon \boldsymbol{v}$ is a probability vector different from $(1/q)\boldsymbol{1}$ that is a solution to (\ref{matrixprob}).

To finish off the construction, choose an appropriate $\varepsilon$ of the form $1/q^j$.  Write
\begin{equation}
 \frac{1}{q}\boldsymbol{1}+\varepsilon \boldsymbol{v} = \frac{1}{q^j}\left(\begin{array}{c}
 m_1 \\
 m_2 \\
 \vdots\\
 m_q
\end{array}
\right)
\end{equation}
where $m_1+\cdots+m_q = q^j$ and not all $m_t$'s equal $q^{j-1}$.  Label the finite field as $\mathbb{F}_q = \{\beta_1,\cdots,\beta_q\}$.  Construct any function $\mathbb{F}_q^j\to \mathbb{F}_q$ such that in its output table (range) $\beta_t$ appears $m_t$ times.  Let $F(X_1,\cdots,X_j)$ be the polynomial with coefficients in $\mathbb{F}_q$ that represents such function.  The polynomial $F(X_1,\cdots,X_j)$ always exists and it is known as the {\em algebraic normal form} of the function.  Observe that 
\begin{equation}
 \s_{\mathbb{F}_q}(F;X) = \sum_{\beta \in \mathbb{F}_q} N_{\mathbb{F}_q}(F;\beta)X^\beta = \sum_{t=1}^q m_t X^{\beta_t}.
\end{equation}
and so $F({\bf X})$ is not balanced over $\mathbb{F}_q$.  By assumption, $\boldsymbol{e}_{n,k}$ is not asymptotically balanced, but 
\begin{eqnarray*}
 \G_{k;F}^{(q)}(X) &=& \G_{k}^{(q)}(X)\cdot \frac{1}{q^j}\s_{\mathbb{F}_q}(F;X) = \sum_{\beta\in\mathbb{F}_q} \frac{1}{q} X^\beta
\end{eqnarray*}
by construction of $F$.  Therefore, $\boldsymbol{e}_{n,k}$ is not asymptotically balanced over $\mathbb{F}_q$, $F({\bf X})$ is not balanced over $\mathbb{F}_q$, but $\boldsymbol{e}_{n,k}+F({\bf X})$ is asymptotically  balanced over $\mathbb{F}_q$.

\begin{remark}
We know that Proposition \ref{PropAsympBalanced} is true when $q=p$. Therefore, the construction will fail to produce a counterexample over $\mathbb{F}_p$.  The step that fails is $A_{p,k}$ being singular.  See, when $p$ is prime, the matrix $A_{p,k}$ is not only doubly stochastic, but also a circulant matrix.  Therefore, its determinant will be given by 
\begin{equation}
 \det(A_{p,k}) = \prod_{t=0}^{p-1}\left(a_0+a_{p-1}\omega_t+a_{p-2}\omega_t^2+\cdots+a_1 \omega_t^{p-1}\right),
\end{equation}
where $\omega_t=\exp(2\pi i t/p)$.  But then $\det(A_{p,k})=0$ if and only if $a_0=a_1=\cdots=a_{p-1}$, i.e. if and only if $\boldsymbol{e}_{n,k}$ is asymptotically balanced over $\mathbb{F}_p$.  However, $\boldsymbol{e}_{n,k}$ was specifically chosen to be not asymptotically balanced. 
\end{remark}

\begin{example}
Consider $\mathbb{F}_4=\mathbb{F}_2(\alpha)$ with $\alpha^2+\alpha+1=0$.  Select $\boldsymbol{e}_{n,3}$ and observe that 
\begin{equation}
 \G_{3}^{(4)} =\frac{5}{16}+\frac{5}{16}X+\frac{3}{16}X^{\alpha}+\frac{3}{16}X^{\alpha +1}.
\end{equation}
For this particular example, the $4\times4$ matrix $A_{4,3}$ is given by 
\begin{equation}
A_{4,3}=\left(
\begin{array}{cccc}
 5/16 & 5/16 & 3/16 & 3/16 \\
 5/16 & 5/16 & 3/16 & 3/16 \\
 3/16 & 3/16 & 5/16 & 5/16 \\
 3/16 & 3/16 & 5/16 & 5/16 \\
\end{array}
\right),
\end{equation}
which is a singular $4\times 4$ doubly stochastic matrix.  Therefore $(1/4)\boldsymbol{1}$ is an eigenvector for $A_{4,3}$.

The null space of $A_{4,3}$ is spanned by the vectors
$$\boldsymbol{v}_1=\left(
\begin{array}{r}
 -1 \\
 1 \\
 0 \\
 0 \\
\end{array}
\right)\,\, \text{ and }\,\,\boldsymbol{v}_2=\left(
\begin{array}{r}
 0 \\
 0 \\
 -1 \\
 1 \\
\end{array}
\right).$$
Observe that the entries of both vectors add up to 0, as predicted by the above discussion.  Choose $\varepsilon=1/4^2=1/16$.  Then,
$$\frac{1}{4}\boldsymbol{1}+\varepsilon \boldsymbol{v}_1=\left(
\begin{array}{c}
 3/16 \\
 5/16 \\
 1/4 \\
 1/4 \\
\end{array}
\right)=\frac{1}{16}\left(
\begin{array}{c}
 3 \\
 5 \\
 4 \\
 4 \\
\end{array}
\right).$$
Now choose a polynomial in two variables over $\mathbb{F}_4$ such that it returns the value 0 three times, the value 1 five times, the value $\alpha$ four times and the value $\alpha+1$ four times.  Such polynomials exist and 
$$F(X_1,X_2) = X_1^3 X_2^3+X_2^2+X_1^2$$
is an example. 

Note that 
$$\s_{\mathbb{F}_4}(F;X)=3+5X+4X^\alpha+4X^{\alpha+1}$$
and
\begin{eqnarray}
 \G_{3;F}^{(4)}(X)&=&\G_{3}^{(4)}(X)\cdot \frac{1}{16}\s_{\mathbb{F}_4}(F;X)\\\nonumber
 &=& \left(\frac{5}{16}+\frac{5}{16}X+\frac{3}{16}X^{\alpha}+\frac{3}{16}X^{\alpha +1}\right)\left(\frac{3}{16}+\frac{5}{16}X+\frac{4}{14}X^\alpha+\frac{4}{16}X^{\alpha+1}\right)\\\nonumber
 &=& \frac{1}{4}+\frac{1}{4}X+\frac{1}{4}X^\alpha+\frac{1}{4}X^{\alpha+1}.
\end{eqnarray}
Therefore, $\boldsymbol{e}_{n,3}+F({\bf X})$ is asymptotically balanced over $\mathbb{F}_4$ even though $\boldsymbol{e}_{n,3}$ is not asymptotically balanced and $F({\bf X}) =  X_1^3 X_2^3+X_2^2+X_1^2$ is not balanced over $\mathbb{F}_4$.  This proves that Proposition \ref{PropAsympBalanced} is not true in general.
\end{example}

\begin{example}
 With $\mathbb{F}_4$ as in the previous example.  Consider $\boldsymbol{e}_{n,9}$ and observe that
 \begin{equation}
  \G_9^{(4)}(X)=\frac{45}{128}+\frac{29}{128}X+\frac{27}{128}X^{\alpha}+\frac{27}{128}X^{\alpha +1}.
 \end{equation}
That implies that $A_{4,9}$ is non-singular and so (\ref{matrixprob}) has only the trivial solution.  We conclude that a perturbation $\boldsymbol{e}_{n,9}+F({\bf X})$ is asymptotically balanced if and only if $F({\bf X})$ is balanced over $\mathbb{F}_4$.
\end{example}




\section{Concluding remarks}

In this article we studied the asymptotic behavior of exponential sums of elementary symmetric polynomials and their perturbations over finite fields. One of the purposes of doing so was to explore the veracity of Conjecture \ref{clsconjFq}.  We extended most of the results that appear in \cite{cm2} to arbitrary finite fields.  We also linked the asymptotic behavior of exponential sums of elementary symmetric polynomials and their perturbations to the value distribution of these polynomials over finite fields.  The concept of asymptotically balanced symmetric polynomial (or perturbation) was also extended to general finite fields.  In the particular case of a perturbation $\boldsymbol{e}_{n,k}+F({\bf X})$, we showed that it is asymptotically balanced over $\mathbb{F}_p$ ($p$ prime) if and only if $\boldsymbol{e}_{n,k}$ is asymptotically balanced or $F({\bf X})$ is balanced over $\mathbb{F}_p$.  We also show that this result does not hold in finite fields in general and provided a way to construct counterexamples.

The asymptotic behavior of the exponential sums considered in this work is dominated by a counting problem over a $q-1$-hypercube of length a power of $p$.  Working on this problem over general finite fields can be difficult and counterintuitive.  For example, consider $\boldsymbol{e}_3({\bf X})$ in $\mathbb{F}_8 = \mathbb{F}_2(\alpha)$, with $\alpha^3+\alpha+1=0$.  Then,
 \begin{equation}
  \G_3^{(8)}(X) = \sum_{\beta \in \mathbb{F}_8}\frac{1}{8} X^{\beta},
 \end{equation}
i.e. $\mathbbm{p}^{(8)}_3(\beta) = 1/8$ for every $\beta \in \mathbb{F}_8$. 
That is quite surprising given that for $\boldsymbol{e}_3({\bf X})$ we have $\mathbbm{p}^{(q)}_{3}(0)>\mathbbm{p}^{(q)}_{3}(\beta)$, $\beta\neq 0$ for $\mathbb{F}_2$ and $\mathbb{F}_4$. 
Furthermore, $\boldsymbol{e}_3({\bf X})$ is asymptotically balanced over $\mathbb{F}_8$, but in this case the degree of the elementary polynomial is not of the form $k=dp^l$ with $1\leq d \leq p-1$.  
Moreover, for the first seven elementary symmetric polynomials, i.e. for $1\leq k \leq 7$, we have
\begin{equation}
\G^{(8)}_{k}(X)=\begin{cases}
  \sum_{\beta \in \mathbb{F}_8}\frac{1}{8} X^{\beta} & k\neq 5, 7\\
 \frac{71}{512}+\sum_{\beta \in \mathbb{F}_8^{\times}}\frac{63}{512} X^{\beta} & k=5\\
 \frac{67}{512}+\frac{67}{512}X+\sum_{\beta\in \mathbb{F}_8\setminus\mathbb{F}_2}\frac{63}{512} X^{\beta} & k=7.
\end{cases}
\end{equation}
This example provides further evidence about the difficulty to determine the veracity of the generalized conjecture of Cusick, Li and St$\check{\mbox{a}}$nic$\check{\mbox{a}}$.
\medskip

\noindent
{\bf Acknowledgments.}  The authors would like to thank Oscar E. Gonz\'alez for reading a previous version of this article.  

\bibliographystyle{plain}

\begin{thebibliography}{1}

\bibitem{aberth} O. Aberth.  The elementary symmetric functions in a finite field of primer order.  
\newblock {\it Illinois J. Math.} {\bf 8(1)} (1964), 132--138.

\bibitem{acgmr} R. A. Arce-Nazario, F. N. Castro, O. E. Gonz\'alez, L. A. Medina, and I. M. Rubio.
\newblock New families of balanced symmetric functions and a generalization of Cuscik, Li and
P. St$\check{\mbox{a}}$nic$\check{\mbox{a}}$.
\newblock {\it Designs, Codes and Cryptography} {\bf 86} (2018), 693--701.

\bibitem{cai} J. Cai, F. Green and T. Thierauf. 
\newblock On the correlation of symmetric functions.
\newblock {\it Math. Systems Theory} {\bf 29} (1996) 245--258.

\bibitem{canteaut} A. Canteaut and M. Videau.
\newblock Symmetric Boolean Functions.
\newblock {\it IEEE Trans. Inf. Theory} {\bf 51(8)} (2005) 2791--2881.

\bibitem{cgm2} F. N. Castro, O. E. Gonz\'alez, and L. A. Medina.
\newblock Diophantine equations with binomial coefficients and perturbations of symmetric Boolean functions.
\newblock {\it IEEE Trans. Inf. Theory} {\bf  64(2)} (2018) 1347--1360.

\bibitem{cm1} F. N. Castro and L. A. Medina. 
\newblock Linear Recurrences and Asymptotic Behavior of Exponential Sums of Symmetric Boolean Functions. 
\newblock {\it Elec. J. Combinatorics} {\bf 18} (2011) \#P8.

\bibitem{cm2} F. N. Castro and L. A. Medina. 
\newblock Asymptotic Behavior of Perturbations of Symmetric Functions.  
\newblock {\it Annals of Combinatorics} {\bf 18} (2014) 397--417.

\bibitem{cm3} F. N. Castro and L. A. Medina. 
\newblock Modular periodicity of exponential sums of symmetric Boolean functions.
\newblock {\it Discrete Appl. Math.} {\bf 217} (2017) 455--473.

\bibitem{cms} F. N. Castro, L. A. Medina, and P. St$\check{\mbox{a}}$nic$\check{\mbox{a}}$.
\newblock Generalized Walsh transforms of symmetric and rotation symmetric Boolean functions are linear recurrent.
\newblock {\it Appl. Algebra Eng. Commun. Comput.} {\bf 29(5)} (2018) 433--453.

\bibitem{ccms} F. N. Castro, R. Chapman, L. A. Medina, and L. B. Sep\'ulveda.  
\newblock Recursions associated to trapezoid, symmetric and rotation symmetric functions over Galois fields.
\newblock {\it Discrete Mathematics}, {\bf 341(7)} (2018) 1915--1931.

\bibitem{cmsep} F. N. Castro, L. A. Medina, and L. B. Sep\'ulveda.
\newblock Closed formulas for exponential sums of symmetric polynomials over Galois fields
\newblock {\it J. Algebr. Comb.} {\bf 50(1)} (2019) 73-98.

\bibitem{cusick4} T. W. Cusick. 
\newblock Hamming weights of symmetric Boolean functions.
\newblock {\it Discrete Appl. Math.} {\bf 215} (2016) 14--19.

\bibitem{cusick2} T. W. Cusick, Y. Li, and  P. St$\check{\mbox{a}}$nic$\check{\mbox{a}}$.
\newblock Balanced Symmetric Functions over $GF(p)$.
\newblock {\it IEEE Trans. Inf. Theory} {\bf 54 (3)} (2008) 1304--1307.

\bibitem{fengliu} K. Feng and F. Liu. 
\newblock New Results On The Nonexistence of Generalized Bent Functions. 
\newblock {\it IEEE Trans. Inf. Theory} {\bf 49} (2003) 3066--3071.

\bibitem{fine} N. J. Fine.
\newblock On the asymptotic distribution of the elementary symmetric functions (mod $p$).
\newblock {\it Trans. Amer. Math. Soc.} {69(1)} (1950), 109--129.

\bibitem{ggz} G. Gao, Y. Guo, and Y. Zhao.
\newblock Recent Results on Balanced Symmetric Boolean Functions.
\newblock {\it IEEE Trans. Inf. Theory} {\bf 62 (9)} (2016) 5199--5203.

\bibitem{hg} Y. Hu and G. Xiao.
\newblock Resilient Functions Over Finite Fields.
\newblock {\it IEEE Trans. Inf. Theory} {\bf 49} (2003) 2040--2046.

\bibitem{ibe} O. Ibe. 
\newblock Markov Processes for Stochastic Modeling (Elsevier Insights),
\newblock Second Edition (2013), Elsevier, Boston, MA.

\bibitem{ims} E. J. Iona\c{s}cu, T. Martinsen, and P. St$\check{\mbox{a}}$nic$\check{\mbox{a}}$.
\newblock Bisecting binomial coefficients.
\newblock {\it Discrete Appl. Math.} {\bf 227} (2017) 70--83.

\bibitem{KScW} P.V. Kumar, R.A. Scholtz, and L.R. Welch. 
\newblock Generalized Bent Functions and Their Properties.
\newblock {\it J. Combinatorial Theory (A)}, {\bf 40} (1985) 90--107.

\bibitem {licusick1} Y. Li and T.W. Cusick. 
\newblock Linear Structures of Symmetric Functions over Finite Fields.
\newblock {Inf. Processing Letters} {\bf 97} (2006) 124--127.

\bibitem{licusick2} Y. Li and T. W. Cusick. 
\newblock Strict Avalanche Criterion Over Finite Fields.
\newblock {\it J. Math. Cryptology}  {\bf 1(1)} (2007) 65--78.

\bibitem{llm}M. Liu, P. Lu and G.L. Mullen. 
\newblock Correlation-Immune Functions over Finite Fields.
\newblock {\it IEEE Trans. Inf. Theory} {\bf 44} (1998), 1273--1276.

\bibitem{mitchell}  C. Mitchell.
\newblock Enumerating Boolean functions of cryptographic significance.
\newblock {\em J. Cryptology} {\bf 2(3)} (1990) 155--170.

\bibitem{rp} C. Riera and M. G. Parker.
\newblock Generalized bent criteria for Boolean functions.
\newblock {\it IEEE Trans. Inform. Theory} {\bf 52(9)}  (2006) 4142--4159.

\bibitem{jdsmith} J. D. Smith.
\newblock Probability and the elementary symmetric functions.
\newblock {\em Proc. Camb. Phil. Soc.} {\bf 74} (1973) 133--139.

\end{thebibliography}

\end{document}